\documentclass{amsart}
\numberwithin{equation}{section}
\theoremstyle{definition}
\newtheorem{thm}{Theorem}[section]
\newtheorem{prop}[thm]{Proposition}

\newtheorem{lem}[thm]{Lemma}
\newtheorem{rem}[thm]{Remark}
\newtheorem{cor}[thm]{Corollary}

\newtheorem*{ack}{Acknowledgments}
\newtheorem{mt}{Main Theorem}
\def\Aut{\mathop{\mathrm{Aut}}\nolimits}

\def\Proj{\mathop{\mathrm{Proj}}\nolimits}
\def\Sym{\mathop{\mathrm{Sym}}\nolimits}

\title[Weighted Galois points for $K3$ surfaces]
{Weighted Galois points for $K3$ surfaces in $\mathbb{P}(1,1,1, 3)$}
\author[S.~Taki]{Shingo Taki}
\address{Department of Mathematics, Tokai University,
4-1-1, Kitakaname, Hiratsuka, Kanagawa, 259-1292, Japan}
\email{staki@tokai.ac.jp}
\urladdr{https://taki.sm.u-tokai.ac.jp}
\date{\today}
\dedicatory{}
\thanks{}
\begin{document}

\begin{abstract}
This paper introduces a generalization of classical Galois points by 
extending the ambient space from standard projective spaces to weighted projective spaces. 
Specifically, the study focuses on smooth weighted hypersurfaces of degree 6 in $\mathbb{P}(1,1,1,3)$
 (which are $K3$ surfaces) to determine the defining equations and the number of weighted Galois points. 
 Furthermore, it characterizes these $K3$ surfaces with weighted Galois points in terms of their automorphisms.
\end{abstract}

\maketitle


\section{Introduction}\label{Introduction}
In algebraic geometry, hypersurfaces of a projective space are among the most fundamental objects.
Galois points, defined by point projections and function fields,  
have been defined for such hypersurfaces in projective spaces.
While various authors have studied many generalizations of Galois points, 
there is no known generalization from the perspective of the ambient space.
In this paper, we present a generalization concerning Galois points.  
That is, we extend the ambient space to weighted projective spaces.
Weighted projective spaces have long appeared implicitly in algebraic geometry.
One of the most basic examples is a hyperelliptic curve 
$y^{2} = f_{2g+2}(x)$ viewed as a double cover of $\mathbb{P}^{1}$ branched at $2g+2$ points. 
Such a curve is a weighted hypersurface of degree $2g+2$ in $\mathbb{P}(1,1, g+1)$. 

In this paper, we introduce a weighted analogue of Galois points for
hypersurfaces in weighted projective spaces.  
Roughly speaking, we replace the usual linear projection from a point 
by a weighted coordinate projection, 
i.e., a rational map obtained by forgetting one weighted homogeneous coordinate 
after a suitable weighted projective change of coordinates. 
The precise definition is given in Section \ref{tentohouteishiki}.
This should be regarded as a weighted analogue of the usual projection
from a point in ordinary projective space.

Yoshihara \cite{Yoshihara-GK3} determined defining equations for 
quartic surfaces containing Galois points, and the possible number of such points.
Our first purpose is to establish a weighted analogue of this classification 
for hypersurfaces of degree 6 in $\mathbb{P}(1,1,1,3)$, i.e., 
for double covers of $\mathbb{P}^{2}$ branched along non-singular sextic curves.
In addition to determining the possible numbers of weighted Galois points, 
we show that their existence induces strong restrictions on the defining
equation of the $K3$ surface.

\begin{mt}\label{mt1} 
Let $S$ be a smooth weighted hypersurface of degree 6 in $\mathbb{P}(1,1,1, 3)$ defined by $Y^{2}=F_{6}(X_{0}, X_{1}, X_{2})$.
The following statements hold:
\begin{enumerate}
\item $S$ has at least one weighted Galois point.
\item The number of weighted Galois points in $S$ is either 0 or 1.
If $S$ has such a point then 
$F_{6}(X_{0}, X_{1}, X_{2})=X_{0}^{5}F_{1}(X_{1}, X_{2})+F_{6}(X_{1}, X_{2})$
holds up to projective equivalence.
\item The number of weighted Galois points not included in $S$ is 1, 2 or 4.
If $S$ has two or more such points then 
$F_{6}(X_{0}, X_{1}, X_{2})=X_{0}^{6}+F_{6}(X_{1}, X_{2})$
holds up to projective equivalence.
\end{enumerate}
Here $F_{d}$ is a homogeneous polynomial of degree $d$.
\end{mt}

There is another aspect of weighted Galois points which is specific to the geometry of $K3$ surfaces. 
Recently, we (\cite{MT1}, \cite{MT2}, \cite{MT3}) obtained a characterization of 
quartic surfaces with a (quasi-)Galois point in terms of $K3$ surfaces with automorphisms. 
The results above suggest that the same philosophy should persist in the weighted setting.

Our second purpose is to show that this is indeed the case, and moreover
that weighted Galois points can be characterized in terms of automorphisms of $K3$ surfaces.  

\begin{mt}\label{mt2} 
The following statements hold:
\begin{enumerate}
\item If a smooth weighted hypersurface of degree 6 in $\mathbb{P}(1,1,1, 3)$ has a weighted Galois point,
then there exists an involution whose fixed locus consists of one smooth curve of genus 10.
Conversely, if a $K3$ surface has an involution whose fixed locus consists of one smooth curve of genus 10
then by a suitable coordinate change, it can be embedded in $\mathbb{P}(1,1,1, 3)$ 
as a smooth weighted hypersurface of degree 6 with a weighted Galois point.

\item If a smooth weighted hypersurface of degree 6 in $\mathbb{P}(1,1,1, 3)$ has a weighted Galois point
included in it, then there exists an automorphism of order 5 
whose fixed locus consists of one smooth curve of genus 2 and one isolated point.
Conversely, if a $K3$ surface has an automorphism of order 5 
whose fixed locus consists of one smooth curve of genus 2 and one isolated point
then by a suitable coordinate change, it can be embedded in $\mathbb{P}(1,1,1, 3)$ 
as a smooth weighted hypersurface of degree 6 with a weighted Galois point included in the hypersurface.

\item If a smooth weighted hypersurface of degree 6 in $\mathbb{P}(1,1,1, 3)$ has at least 
two weighted Galois points that are not included in it, then there exists an automorphism of order 6 
whose fixed locus consists of one smooth curve of genus 2.
Conversely, if a $K3$ surface has an automorphism of order 6 
whose fixed locus consists of one smooth curve of genus 2
then by a suitable coordinate change, it can be embedded in $\mathbb{P}(1,1,1, 3)$ 
as a smooth weighted hypersurface of degree 6 with at least two weighted Galois points 
that are not included in the hypersurface.
\end{enumerate}
\end{mt}

We summarize the contents of this paper. 
In Section \ref{tentohouteishiki}, we study smooth weighted hypersurfaces
given by $Y^{2}=F_{6}(X_{0}, X_{1}, X_{2})$, i.e.,  
double covers of $\mathbb{P}^{2}$ branched along a nonsingular sextic curve.
Then we give proofs of Main Theorem \ref{mt1} and the first half of Main Theorem \ref{mt2}. 
In Section \ref{K3toGalois}, we prove the second half (after ``Conversely'') of Main Theorem \ref{mt2}
by using automorphisms of $K3$ surfaces.
Section \ref{curve} is an appendix.
We will discuss weighted Galois points for smooth curves of genus 1 or 2.

\begin{ack}
The author would like to express his gratitude to Professor Kei Miura for many helpful discussions and suggestions. 
This work was supported by JSPS KAKENHI Grant Number JP23K03036.
\end{ack}

\section{Weighted Galois points and equations}\label{tentohouteishiki}

Let $[X_{0}: X_{1}: X_{2}: Y]$ be the weighted homogeneous coordinates
on the weighted projective space 
$\mathbb{P}(1,1,1, 3)=\Proj \mathbb{C}[X_0,X_1,X_2,Y]$, 
hence $\deg X_{i}=1$ and $\deg Y=3$.
We first specify what we mean by a projection from a point of $\mathbb{P}(1,1,1, 3)$.

For a point $P_{Y}:=[0:0:0:1]$, the rational map
\[  \rho_Y: \mathbb{P}(1,1,1, 3) \dashrightarrow \mathbb{P}^{2},
 \quad
 [X_{0} : X_{1} : X_{2} : Y] \mapsto [X_{0} : X_{1} : X_{2}],\]
will be called the standard weighted projection centered at $P_{Y}$.
Similarly, we call
\[  \rho_{0}:\mathbb{P}(1,1,1, 3) \dashrightarrow\mathbb P(1,1,3),
 \quad
 [X_{0} : X_{1} : X_{2} : Y] \mapsto [X_{1} : X_{2} : Y] \]
the standard weighted projection centered at $P_{0}:=[1:0:0:0]$. 
The projections $\rho_{1}$ and $\rho_{2}$ are defined analogously.

More generally,  by a weighted projection centered at $P \in\mathbb{P}(1,1,1, 3)$, 
we mean a rational map obtained from one of the standard weighted projections 
by a weighted projective change of coordinates sending $P$ to its center. 
Thus, if an automorphism $\varphi$ of $\mathbb{P}(1,1,1, 3)$ satisfies $\varphi(P)=P_{i}$, then
\[  \rho_{i} \circ \varphi : \mathbb{P}(1,1,1, 3) \dashrightarrow \mathbb{P}(1,1,3) \]
is a weighted projection centered at $P$; if $\varphi(P)=P_{Y}$, we use $\rho_{Y}\circ\varphi$ instead.

Let $S$ be a weighted hypersurface in $\mathbb{P}(1,1,1, 3)$ and
$\pi_{P}:S\dashrightarrow H$ a weighted projection centered at $P$
whose restriction to $S$ is dominant and generically finite.
Note that the projection $\pi_{P}$ induces the extension 
$\mathbb{C}(S)/\pi_{P}^{\ast}\mathbb{C}(H)$ of function fields.
The point $P$ is called a \textit{weighted Galois point} for $S$
if there exists a weighted projection $\pi_{P}^{\ast}$ inducing 
the Galois extension $\mathbb{C}(S)/\pi_{P}^{\ast}\mathbb{C}(H)$. 

\begin{rem}
Every point of $\mathbb{P}(1,1,1, 3)$ admits a weighted projection in the above sense.  
We consider a point $P=[a_{0} : a_{1}: a_{2}: b]\neq P_{Y}$.
Since $(a_{0}, a_{1},a_{2})\neq(0,0,0)$ holds, 
we may assume that $P=[1:0:0:c]$ after a linear change of the variables $X_{0}, X_{1}, X_{2}$.
Replacing $Y$ by $Y-cX_{0}^{3}$, which is a weighted homogeneous change of
coordinates, we obtain $P=[1:0:0:0]$.
Thus every point distinct from $P_{Y}$ can be used as the center of a
weighted projection to $\mathbb{P}(1,1,3)$ after a suitable weighted
projective change of coordinates.
\end{rem}

Let $S$ be a smooth weighted hypersurface of degree 6 in $\mathbb{P}(1,1,1, 3)$.
Note that while $S$ is generally defined as
$Y^{2}+YF_{3}(X_{0}, X_{1}, X_{2})+F_{6}(X_{0}, X_{1}, X_{2})$,
if we replace the coordinates with $Y'=Y+F_{3}(X_{0}, X_{1}, X_{2})/2$,
it is of the form $S:Y'^{2}+F_{6}(X_{0}, X_{1}, X_{2})=0$. 
Here $F_{d}$ is a homogeneous polynomial of degree $d$.

In the following, we assume that $S$ is given by
\[ S:Y^{2}=F_{6}(X_{0}, X_{1}, X_{2}). \]
This is a $K3$ surface, which is a double covering of the projective plane branched along 
a nonsingular sextic curve $F_{6}(X_{0}, X_{1}, X_{2})=0$.

A weighted Galois point $P$ is called \textit{inner} (resp. \textit{outer}) 
if $P \in S$ (resp. $P \not \in S$).

\begin{prop}\label{gm1}
The following statements hold:
\begin{enumerate}
\item The point $[0:0:0:1]$ is an outer weighted Galois point for $S$.
\item If $F_{6}(X_{0}, X_{1}, X_{2})=X_{0}^{5}F_{1}(X_{1}, X_{2})+F_{6}(X_{1}, X_{2})$ holds
then the point $[1:0:0:0]$ is an inner weighted Galois point for $S$.
\item If $F_{6}(X_{0}, X_{1}, X_{2})=X_{0}^{6}+F_{6}(X_{1}, X_{2})$ holds
then the point $[1:0:0:0]$ is an outer weighted Galois point for $S$.
\end{enumerate}
\end{prop}
\begin{proof}
\begin{enumerate}
\item  We take a weighted projection 
\[\pi_{P}:\mathbb{P}(1, 1, 1, 3) \dashrightarrow \mathbb{P}^{2}, \ \ [X_{0}: X_{1}: X_{2}: Y] \mapsto [X_{0}: X_{1}: X_{2}]\]
from $P:=[0:0:0:1]$.
Put $x_{1}:=X_{1}/X_{0}$, $x_{2}:=X_{2}/X_{0}$, $y:=Y/X_{0}^{3}$, and 
$f_{6}(x_{1}, x_{2})=F_{6}(1, X_{1}/X_{0}, X_{2}/X_{0})$.
Since the function field $\mathbb{C}(S)$ of $S$ is the quotient field of 
the ring $\mathbb{C}[x_{1}, x_{2}]/(y^{2}-f_{6}(x_{1}, x_{2}))$,
it is $\mathbb{C}(x_{1}, x_{2})(y)$ satisfying $y^{2}=f_{6}(x_{1}, x_{2})$.
We consider the extension 
\[ \mathbb{C}(S)/\pi_{P}^{\ast}\mathbb{C}(\mathbb{P}^{2})=
Q \left( \mathbb{C}[x_{1}, x_{2}]/(y^{2}-f_{6}(x_{1}, x_{2})) \right) /\mathbb{C}(x_{1}, x_{2}).\]
We remark that $T^{2}-f_{6}(x_{1}, x_{2}) \in \mathbb{C}(x_{1}, x_{2})[T]$ is the minimal polynomial of $y$ over $\mathbb{C}(x_{1}, x_{2})$,
and $\mathbb{C}(S)$ includes the other root $-y$.
Thus, the extension is Galois.

\item We take a weighted projection 
\[\pi_{Q}:\mathbb{P}(1, 1, 1, 3) \dashrightarrow \mathbb{P}(1, 1,  3), \ \ [X_{0}: X_{1}: X_{2}: Y] \mapsto [X_{1}: X_{2}:Y]\]
from $Q:=[1:0:0:0]$.
Put $x_{0}:=X_{0}/X_{2}$, $x_{1}:=X_{1}/X_{2}$, $y:=Y/X_{2}^{3}$.
Since the function field of $S$ is $\mathbb{C}(x_{1}, y)(x_{0})$ satisfying $y^{2}=x_{0}^{5}f_{1}(x_{1})+f_{6}(x_{1})$,
and $\pi_{P}^{\ast}\mathbb{C}(\mathbb{P}(1, 1,  3)) \simeq \mathbb{C}(x_{1}, y)$, the polynomial
\[ T^{5}-\frac{y^{2}-f_{6}(x_{1})}{f_{1}(x_{1})} \in \mathbb{C}(x_{1}, y)[T] \]
is the minimal polynomial of $x_{0}$ over $\mathbb{C}(x_{1}, y)$.
The other roots $\zeta_{5}^{i}x_{0} \  (i=1, 2, 3, 4)$ are included in $\mathbb{C}(S)$.
Thus, the extension $\mathbb{C}(S) / \pi_{P}^{\ast}\mathbb{C}(\mathbb{P}(1, 1,  3))$ is Galois.

\item The same argument as above shows the assertion.
\end{enumerate}
\end{proof}

Note that an element $\tau\in G_{P}$ induces an automorphism of the function field 
$\mathbb{C}(S)$ over $\pi_{P}^{\ast}\mathbb{C}(H)$, and hence a birational automorphism of $S$.
Since $S$ is a smooth $K3$ surface, every birational automorphism of $S$ is biregular.
Thus we may regard $G_{P}$ as a subgroup of $\Aut (S)$.

\begin{lem}\label{extend}
Let $P$ be a weighted Galois point for $S$.
The Galois group $G_{P}$ of the Galois extension $\mathbb{C}(S)/\pi_{P}^{\ast}\mathbb{C}(H)$
is a subgroup of $\Aut (\mathbb{P}(1,1,1,3))$.
\end{lem}

\begin{proof}
If $P=[0:0:0:1]$ then the automorphism 
$[X_{0}: X_{1}: X_{2}: Y] \mapsto [X_{0}: X_{1}: X_{2}:-Y]$
is a generator of $G_{P}$ by Proposition \ref{gm1} (1).

Assume $P\neq [0:0:0:1]$.
Since $\mathcal{O}_{S}(1)=\pi_{P}^{\ast}\mathcal{O}_{\mathbb{P}(1,1,3)}(1)$ and 
$\tau \in G_{P}$ satisfies $\pi_{P} \circ \tau=\pi_{P}$, we have 
$\tau^{\ast}\mathcal{O}_{S}(1)=\mathcal{O}_{S}(1)$.
Hence it induces linear automorphisms on $H^{0}(S, \mathcal{O}_{S}(1))=\langle X_{0}, X_{1}, X_{2} \rangle$. 
It also induces linear automorphisms on 
$H^{0}(S, \mathcal{O}_{S}(3))=\Sym^{3}H^{0}(S, \mathcal{O}_{S}(1) \oplus \mathbb{C}Y)$.
Since $\tau^{\ast}$ preserves $\Sym^{3}H^{0}(S, \mathcal{O}_{S}(1))$, 
there exist $A \in GL(3,\mathbb{C})$, $a \in \mathbb{C}^{\ast}$, 
and a homogeneous polynomial $F_{3}(X_{0}, X_{1},X_{2})$ of degree 3 such that
 \[ \tau^{\ast} \begin{pmatrix} X_{0}\\ X_{1}\\ X_{2} \end{pmatrix} 
 = A \begin{pmatrix} X_{0} \\ X_{1} \\ X_{2} \end{pmatrix}, \quad \tau^{\ast}Y = aY+F_{3}(X_{0}, X_{1}, X_{2}). \]
Since these transformations preserve the weights of $X_{0}, X_{1}, X_{2}, Y$,
these define a graded automorphism of $\mathbb{C}[X_{0}, X_{1}, X_{2}, Y]$.
Thus $\tau$ extends to an automorphism of $\mathbb{C}[X_{0}, X_{1}, X_{2}, Y]$, 
meaning $G_{P}$ is a subgroup of $\Aut (\mathbb{P}(1,1,1,3))$.
\end{proof}

\begin{prop}\label{sgg}
Let $P$ be a weighted Galois point for $S$ and $G_{P}$ the Galois group of the Galois extension $\mathbb{C}(S)/\pi_{P}^{\ast}\mathbb{C}(H)$.
\begin{enumerate}
\item If $P=[0:0:0:1]$ then $G_{P}$ is isomorphic to $\mathbb{Z}/2\mathbb{Z}$.
\item If $P\in S$ then $G_{P}$ is isomorphic to $\mathbb{Z}/5\mathbb{Z}$.
\item If $P \not \in S$ and $P\neq [0:0:0:1]$ then $G_{P}$ is isomorphic to $\mathbb{Z}/6\mathbb{Z}$.
\end{enumerate}
\end{prop}
\begin{proof}
\begin{enumerate}
\item It follows from Proposition \ref{gm1} (1).
\item Note that the degree of the extension $[\mathbb{C}(S): \pi_{P}^{\ast}\mathbb{C}(H) ]$ is equal to $\deg \pi_{P}$.
Since $S$ is smooth, if $P\in S$ then $\deg \pi_{P}=5$.
Thus (2) holds.
\item By the same argument as above, the order of $G_{P}$ is 6.
Let $L$ be a line passing through $P$.
Since $\tau \in G_{P}$ preserves $L$ and fixes the intersection point $\infty $ of $L$ and $\mathbb{P}(1,1,3)$,
it acts on the affine line which is $L \setminus \{\infty \}$ as an automorphism.
Thus, we can regard $G_{P}$ as a group that acts on $\mathbb{A}^{1}$ as an automorphism.
Since such a group on $\mathbb{A}^{1}$ is cyclic, $G_{P}$ is isomorphic to $\mathbb{Z}/6\mathbb{Z}$.
\end{enumerate}
\end{proof}

\begin{rem}
The assumption that $S$ is smooth is essential.
Let $S$ be a weighted hypersurface given by 
$X_{0}^{4}F_{2}(X_{1}, X_{2})+F_{6}(X_{1}, X_{2}, Y)=0$
in $\mathbb{P}(1,1,1, 3)$.
Then $P=[1:0:0:0]$ is a weighted inner Galois point, and a singular point of type $A_{1}$.
It is easy to see that the order of $G_{P}$ is 4.
\end{rem}

From the above arguments, we obtain the following.

\begin{cor}\label{eq2}
Let $P \neq [0:0:0:1]$ be a Galois point for $S$.
\begin{enumerate}
\item If $P \in S$ then the defining equation of $S$ can be given by $Y^{2}=X_{0}^{5}F_{1}(X_{1}, X_{2})+F_{6}(X_{1}, X_{2})$.
\item If $P \not \in S$ then the defining equation of $S$ can be given by $Y^{2}=X_{0}^{6}+F_{6}(X_{1}, X_{2})$.
\end{enumerate}
\end{cor}
\begin{proof}
These follow from Lemma \ref{extend} and Proposition \ref{sgg}.
\end{proof}

\begin{prop}\label{number}
The following statements hold:
\begin{enumerate}
\item The number of weighted inner Galois points for $S$ is 0 or 1.
\item The number of weighted outer Galois points for $S$ is 1, 2, or 4.
\end{enumerate}
\end{prop}

\begin{proof}
(1) Assume that $S:Y^{2}=X_{0}^{5}F_{1}(X_{1}, X_{2})+F_{6}(X_{1}, X_{2})$ 
has a weighted inner Galois point $P\neq [1:0:0:0] \in S$.
Since $G_{P} \simeq \mathbb{Z}/5\mathbb{Z}$ acts on $S$ by Proposition \ref{sgg}, 
we may put $S: Y^{2}=X_{0}^{5}X_{2}+X_{1}^{5}X_{2}+X_{2}^{6}$.
Note that points $[1:0:0:0]$ and $[0:1:0:0]$ are weighted Galois points for $S$.
However, $S$ has five $A_{1}$ singular points at $[1:\zeta_{10}^{i}:0:0]$ ($i=1, 3, 5, 7, 9$).

We see that the point $Q=[1:-1:0:0]$ is a singular point of type $A_{1}$.
Put $x_{1}:=X_{1}/X_{0}$, $x_{2}:=X_{2}/X_{0}$, and $y:=Y/X_{0}^{3}$.
Then $Q$ and $S$ are denoted by $Q =(x_{1}, x_{2}, y)= (-1, 0, 0)$ and 
$S:y^{2}=(1+x_{1}^{5}+x_{2}^{5})x_{2}$, respectively.
If we replace the coordinates $u=x_{1}+1$ and $v=x_{2}$, then $Q=(u, v, y)=(0, 0, 0)$.
Thus, the local equation is
\begin{align*}
y^{2}&= (1+x_{1}^{5}+x_{2}^{5})x_{2}\\
&=(1+(-1+5u+O(u^{2})+v^{5}))v\\
&=5uv+(\text{degree 3 or higher}).
\end{align*}
Hence $Q$ is a singular point of type $A_{1}$.

(2) This is essentially the same as (1).
Assume that $S: Y^{2}=X_{0}^{6}+F_{6}(X_{1}, X_{2})$
has a weighted outer Galois point $P$ distinct from $[0:0:0:1]$ and $[1:0:0:0]$.
Since $G_{P} \simeq \mathbb{Z}/6\mathbb{Z}$ acts on $S$ by Proposition \ref{sgg},
we have $S: Y^{2}=X_{0}^{6}+X_{1}^{6}+ X_{2}^{6}$.
In this case, $[0:0:0:1]$, $[1:0:0:0]$, $[0:1:0:0]$, and $[0:0:1:0]$ are weighted outer Galois points.
See also Proposition \ref{gm1}.
\end{proof}

\begin{prop}
The following statements hold:
\begin{enumerate}
\item If a smooth weighted hypersurface of degree 6 in $\mathbb{P}(1,1,1, 3)$ has a weighted Galois point,
then there exists an involution whose fixed locus consists of one smooth curve of genus 10.

\item If a smooth weighted hypersurface of degree 6 in $\mathbb{P}(1,1,1, 3)$ has a weighted Galois point
included in it, then there exists an automorphism of order 5 
whose fixed locus consists of one smooth curve of genus 2 and one isolated point.

\item If a smooth weighted hypersurface of degree 6 in $\mathbb{P}(1,1,1, 3)$ has at least 
two weighted Galois points that are not included in it, then there exists an automorphism of order 6 
whose fixed locus consists of one smooth curve of genus 2.
\end{enumerate}
\end{prop}
\begin{proof}
Since equations for weighted hypersurfaces are given by Proposition \ref{gm1} and Corollary \ref{eq2},
we take a generator of each Galois group.
\begin{enumerate}
\item Since a weighted hypersurface is $Y^{2}=F_{6}(X_{0}, X_{1}, X_{2})$, we find an automorphism
$[X_{0}: X_{1}: X_{2}:Y] \mapsto [X_{0}: X_{1}: X_{2}:-Y]$.
Its fixed locus consists of one smooth curve defined by $F_{6}(X_{0}, X_{1}, X_{2})=0$.

\item Since a weighted hypersurface is $Y^{2}=X_{0}^{5}F_{1}(X_{1}, X_{2})+F_{6}(X_{1}, X_{2})$, 
we find an automorphism $[X_{0}: X_{1}: X_{2}:Y] \mapsto [\zeta_{5}X_{0}: X_{1}: X_{2}:Y]$.
Its fixed locus consists of one smooth curve defined by $Y^{2}=F_{6}(X_{1}, X_{2})$, and 
one point $[1:0:0:0]$.

\item Since a weighted hypersurface is $Y^{2}=X_{0}^{6}+F_{6}(X_{1}, X_{2})$, 
we find an automorphism $[X_{0}: X_{1}: X_{2}:Y] \mapsto [\zeta_{6}X_{0}: X_{1}: X_{2}:Y]$.
Its fixed locus consists of one smooth curve defined by $Y^{2}=F_{6}(X_{1}, X_{2})$.
Note that the point $[1:0:0:0]$ does not lie on the hypersurface.
\end{enumerate}
Here $\zeta_{I}$ is a primitive $I$-th root of unity.
\end{proof}

\section{A viewpoint from $K3$ surfaces with automorphisms}\label{K3toGalois}
In this section, we study weighted Galois points through the lens of $K3$ surfaces and their automorphisms.
We remark that a $K3$ surface $S$ is not necessarily given as a hypersurface of $\mathbb{P}(1, 1, 1, 3)$.
If $S$ has a suitable automorphism, $S$ can be embedded into $\mathbb{P}(1, 1, 1, 3)$
as a weighted hypersurface of degree 6 with a weighted Galois point.
That embedding is a \textit{weighted Galois embedding}, which is a generalization of a Galois embedding \cite{Yoshihara-embedding}.

Let $\omega _{S}$ be a nowhere vanishing holomorphic 2-form on $S$.
An automorphism $\sigma$ of finite order $I$ on $S$ is called \textit{(purely) non-symplectic}  
if it satisfies $\sigma^{\ast} \omega _{S}=\zeta_{I}\omega _{S}$
where $\zeta_{I}$ is a primitive $I$-th root of unity. 
For background on automorphisms $K3$ surfaces, see \cite{Ni, AST, BH}.

\begin{prop}\label{aut-equa}
Let $S$ be a $K3$ surface with a non-symplectic automorphism $\sigma$.
\begin{enumerate}
\item If $\sigma$ is of order 2, and its fixed locus $S^{\sigma}$ consists of one non-singular curve of genus 10,
then $S$ is given by $Y^{2}=F_{6}(X_{0}, X_{1}, X_{2})$ in $\mathbb{P}(1,1,1,3)$ after a suitable choice of coordinates.
\item If $\sigma$ is of order 5, and its fixed locus $S^{\sigma}$ consists of one non-singular curve of genus 2 and one isolated point, 
then $S$ is given by $Y^{2}=X_{0}^{5}F_{1}(X_{1}, X_{2})+F_{6}(X_{1}, X_{2})$ in $\mathbb{P}(1,1,1,3)$ after a suitable choice of coordinates.
\item If $\sigma$ is of order 6, and its fixed locus $S^{\sigma}$ consists of one non-singular curve of genus 2, 
then $S$ is given by $Y^{2}=X_{0}^{6}+F_{6}(X_{1}, X_{2})$ in $\mathbb{P}(1,1,1,3)$ after a suitable choice of coordinates.
\end{enumerate}
\end{prop}

\begin{proof}
\begin{enumerate}
\item The rank of the invariant lattice $L(\sigma)=\{x \in H^{2}(S, \mathbb{Z}) \mid \sigma^{\ast}(x)=x\}$ is 1,
and its generator $h$ satisfies $h^{2}=2$ by \cite[Theorem 4.2.2]{Ni}.
The linear system $|h|$ defines a morphism $\pi :S \to \mathbb{P}^{2}$ of degree 2, whose branch locus
is a sextic curve in $\mathbb{P}^{2}$.
In the following, we study a weighted Galois embedding.

We consider the section ring $R(S, h):=\bigoplus_{m\geq 0}H^{0}(S,\mathcal{O}_{S}(mh))$.
Note that $\dim H^{0}(S,\mathcal{O}_{S}(mh))=m^{2}+2$ holds by the Riemann-Roch theorem and the Kodaira vanishing theorem.
Since the dimension of the symmetric power $\Sym^{3}(H^{0}(S,\mathcal{O}_{S}(h)))$ is 10, 
there exists $Y \in H^{0}(S,\mathcal{O}_{S}(3h))$ that cannot be expressed in terms of the basis
$\{X_{0}, X_{1}, X_{2}\}$ of $H^{0}(S,\mathcal{O}_{S}(h))$.

It is easy to see that the dimension of
the part $\mathbb{C}[X_{0}, X_{1}, X_{2}, Y]_{m}$ of degree $m (\leq 6)$ in $\mathbb{C}[X_{0}, X_{1}, X_{2}, Y]$
is given by
\[ \dim \mathbb{C}[X_{0}, X_{1}, X_{2}, Y]_{m}
=\begin{cases}
\binom{m+2}{2}+\binom{m-1}{2} & m \leq 5\\
\binom{m+2}{2}+\binom{m-1}{2}+1 &m=6
\end{cases}.\]
Since $\dim \mathbb{C}[X_{0}, X_{1}, X_{2}, Y]_{6}-\dim H^{0}(S,\mathcal{O}_{S}(6h))=1$ holds, 
we have a relation $Y^{2}=F_{6}(X_{0}, X_{1}, X_{2})$, hence
we have $R(S, h) \simeq \mathbb{C}[X_{0}, X_{1}, X_{2},Y]/(Y^{2}-F_{6}(X_{0}, X_{1}, X_{2}))$.
This implies that 
\[ S=\Proj R(S, h) \subset \Proj \mathbb{C}[X_{0}, X_{1}, X_{2},Y]=\mathbb{P}(1, 1, 1, 3). \]
This is a weighted Galois embedding.

\item It follows from \cite[P. 520 \fbox{$S(\sigma)\cong H_{5}$}]{AST}.

\item Since there exists a fixed curve $C$ of genus 2, 
the morphism $\pi : S \to \mathbb{P}^{2}$ associated to the linear system $|C|$ is 
a double covering of $\mathbb{P}^{2}$ branched along a smooth plane sextic curve $B$ 
(See also \cite[Proposition VIII.13]{Be}).
Thus, after a suitable choice of coordinates, 
we may assume that $S$ is given by $Y^{2}=F_{6}(X_{0}, X_{1}, X_{2})$ in $\mathbb{P}(1,1,1,3)$.

Since $\sigma$ fixes $C$, it preserves the linear system $|C|$ which defines $\pi$.
Thus $\sigma$ induces a projective transformation $\tilde{\sigma}$ of $\mathbb{P}^{2}$.
Since the fixed locus of $\sigma$ is exactly $C$, 
there exists a line $H$ in $\mathbb{P}^{2}$ such that $\pi^{-1}(H)=C$.
Moreover $\tilde{\sigma}$ fixes $H$ pointwise. 
By replacing coordinates of $\mathbb{P}^{2}$ such that $H=\{X_{0}=0\}$, 
we may assume that 
$\tilde{\sigma}$ satisfies $\tilde{\sigma}([X_{0}:X_{1}:X_{2}])=[\zeta_{6}X_{0}:X_{1}:X_{2}]$.
Thus we find that its equation is of the form $Y^{2}=aX_{0}^{6}+F_{6}(X_{1}, X_{2})$.
The smoothness of $B$ implies $a\neq 0$. After rescaling $X_0$, we therefore obtain
$Y^{2}=aX_{0}^{6}+F_{6}(X_{1}, X_{2})$.
\end{enumerate}
\end{proof}

\begin{rem}
A non-symplectic automorphism $\sigma_{I}$ of order $I$ is characterized in terms of invariants of $L(\sigma_{I})$.
In the cases (1) through (3) above, we have $L(\sigma_{2})=\langle 2 \rangle$, $L(\sigma_{5})=H_{5}$ and 
$L(\sigma_{6})=\langle 2 \rangle$.
\end{rem}

\begin{cor}
The following statements hold:
\begin{enumerate}
\item If a $K3$ surface has a non-symplectic involution whose fixed locus consists of one smooth curve of genus 10
then by a suitable coordinate change, it can be embedded in $\mathbb{P}(1,1,1, 3)$ 
as a smooth weighted hypersurface of degree 6 with a weighted Galois point.

\item If a $K3$ surface has a non-symplectic automorphism of order 5 
whose fixed locus consists of one smooth curve of genus 2 and one isolated point
then by a suitable coordinate change, it can be embedded in $\mathbb{P}(1,1,1, 3)$ 
as a smooth weighted hypersurface of degree 6 with an inner weighted Galois point.

\item If a $K3$ surface has a non-symplectic automorphism of order 6 
whose fixed locus consists of one smooth curve of genus 2
then by a suitable coordinate change, it can be embedded in $\mathbb{P}(1,1,1, 3)$ 
as a smooth weighted hypersurface of degree 6 with at least two outer weighted Galois points.
\end{enumerate}
\end{cor}
\begin{proof}
These follow from Proposition \ref{gm1} and Proposition \ref{aut-equa}.
\end{proof}

\appendix
\section{Weighted Galois points for algebraic curves}\label{curve}

In this section, we treat double covers of $\mathbb{P}^{1}$.
We assume that all the curves which we consider are smooth.

It is well known that the canonical model of an algebraic curve of genus 2 is
given by a double cover of $\mathbb{P}^{1}$ branched at 6 points, i.e., 
weighted hypersurface of degree 6 in $\mathbb{P}(1,1,3)$.
By a similar argument as in Section \ref{tentohouteishiki}, we obtain the following:

\begin{prop}
Let $C$ be a smooth weighted hypersurface of degree 6 in $\mathbb{P}(1,1, 3)$
 defined by $Y^{2}=F_{6}(X_{0}, X_{1})$.
The following statements hold:
\begin{enumerate}
\item $C$ has at least one outer weighted Galois point $[0:0:1]$.
\item The number of inner weighted Galois points for $C$ is either 0 or 1.
If $C$ has such a point then 
$F_{6}(X_{0}, X_{1})=X_{0}^{5}X_{1}+X_{1}^{6}$
holds up to projective equivalence.
Moreover, the inner weighted Galois point is $[1:0:0]$.
\item The number of outer weighted Galois points for $C$ is 1 or 3.
If $C$ has 3 such points then 
$F_{6}(X_{0}, X_{1})=X_{0}^{6}+X_{1}^{6}$
holds up to projective equivalence.
Moreover, the outer weighted Galois points are $[1:0:0]$ and $[0:1:0]$.
\end{enumerate}
Here $F_{d}$ is a homogeneous polynomial of degree $d$.
\end{prop}

An algebraic curve of genus 1 is given as a cubic curve in $\mathbb{P}^{2}$,
but it can also be given as a double cover of $\mathbb{P}^{1}$ branched at 4 points.

\begin{lem}
Let $C$ be a smooth cubic curve in $\mathbb{P}^{2}$.
\begin{enumerate}
\item All points in $C$ are inner Galois points.
\item $C$ has an outer Galois point if and only if the $j$-invariant of $C$ is 0.
\end{enumerate}
\end{lem}
\begin{proof}
\begin{enumerate}
\item Let $\pi_{P}:C \to \mathbb{P}^{1}$ be a projection from a point $P \in C$.
It induces the extension 
$\mathbb{C}(C)/\pi_{P}^{\ast}\mathbb{C}(\mathbb{P}^{1})$ of function fields, and
its degree $[\mathbb{C}(C):\pi_{P}^{\ast}\mathbb{C}(\mathbb{P}^{1})]$ is equal to $\deg \pi_{P}$.
Since the point $P$ lies on $C$, we have $\deg \pi_{P}=2$.
Thus $P$ is a Galois point.

\item Let $G_{P}$ be the Galois group of the Galois extension 
$\mathbb{C}(C)/\pi_{P}^{\ast}\mathbb{C}(\mathbb{P}^{1})$.
If $P$ is an outer Galois point then the order of $G_{P}$ is 3.
Since a generator of $G_{P}$ acts on $C$ as an automorphism,
the $j$-invariant of $C$ is 0. See also \cite[Chapter III, Theorem 10.1]{Sil}.

Conversely, if the $j$-invariant of $C$ is 0 then $C$ is given by 
$X_{1}^{2}X_{2}=X_{0}^{3}+BX_{2}^{3}$.
The point $[1:0:0]$ is an outer Galois point for $C$.
\end{enumerate}
\end{proof}

\begin{prop}
Let $C$ be a smooth weighted hypersurface of degree 4 in $\mathbb{P}(1,1, 2)$
 defined by $Y^{2}=F_{4}(X_{0}, X_{1})$.
The following statements hold:
\begin{enumerate}
\item $C$ does not necessarily have an inner weighted Galois point.
\item $C$ has at least one outer weighted Galois point $[0:0:1]$.
\end{enumerate}
\end{prop}
\begin{proof}
If a point $P$ lies on $C$ then the degree of the field extension
$\mathbb{C}(C)/\pi_{P}^{\ast}\mathbb{C}(\mathbb{P}^{1})$ is equal to $\deg \pi_{P}=3$.
In general, this is not Galois.

The same argument as in Section \ref{tentohouteishiki} shows (2).
\end{proof}

This implies that automorphisms of elliptic curves do not necessarily preserve Galois point data.
This is precisely where the importance of ``Galois embeddings'' lies.

\end{document}